\documentclass[12pt]{amsart}

\newcommand{\bfG}{\mathbf G} %
\newcommand{\bfN}{\mathbf N} %
\newcommand{\bfP}{\mathbf P} %
\newcommand{\bfU}{\mathbf U} %

\renewcommand\d{\mathbf d}
\newcommand\ddim{\mathbf{dim}\:}

\renewcommand{\u}{\mathfrak u} %
\newcommand{\m}{\mathfrak m} %

\DeclareMathOperator\GL{GL} %
\DeclareMathOperator\M{M} %
\DeclareMathOperator\U{U} %
\DeclareMathOperator\Aut{Aut} %
\DeclareMathOperator\End{End} %
\DeclareMathOperator\Hom{Hom} %

\newcommand\ZZ{{\mathbb Z}}
\newcommand\FF{\mathbb F}

\newcommand\CA{\mathcal A}
\newcommand\CC{\mathcal C}
\newcommand\CM{\mathcal M}
\newcommand\CQ{\mathcal Q}
\newcommand\CR{\mathcal R}

\newcommand\sub{\subseteq}

\usepackage{amssymb}
\usepackage{fullpage}

\numberwithin{equation}{section}

\theoremstyle{plain}
\newtheorem{lem}[equation]{Lemma}
\newtheorem{thm}[equation]{Theorem}
\newtheorem{cor}[equation]{Corollary}
\newtheorem{prop}[equation]{Proposition}
\newtheorem{ques}[equation]{Question}

\theoremstyle{remark}
\newtheorem{rem}[equation]{Remark}

\title
{Counting conjugacy classes in the unipotent radical of parabolic
subgroups of $\GL_n(q)$}

\author[S.~M.~Goodwin] {Simon M.~Goodwin}
\address{School of Mathematics,
University of Birmingham,
Birmingham, B15 2TT,
United Kingdom}
\email{goodwin@maths.bham.ac.uk}
\urladdr{http://web.mat.bham.ac.uk/S.M.Goodwin}

\author[G.\ R\"ohrle]{Gerhard R\"ohrle}
\address%[G.~R\"{o}hrle]
{Fakult\"at f\"ur Mathematik, Ruhr-Universit\"at Bochum, D-44780
Bochum, Germany} \email{gerhard.roehrle@rub.de}
\urladdr{http://www.ruhr-uni-bochum.de/ffm
/Lehrstuehle/Lehrstuhl-VI/rubroehrle.html}


\begin{document}

\begin{abstract}
Let $q$ be a power of a prime $p$. Let $P$ be a parabolic subgroup of the
general linear group $\GL_n(q)$ that is the stabilizer of a flag in
$\FF_q^n$ of length at most $5$, and let $U = O_p(P)$.  In this note
we prove that, as a function of $q$, the number $k(U)$ of conjugacy
classes of $U$ is a polynomial in $q$ with integer coefficients.
\end{abstract}

\maketitle

\section{Introduction}
Let $\GL_n(q)$ be the finite general linear group defined over the
field $\FF_q$ of $q$ elements, where $q$ is a power of a prime $p$. A
longstanding conjecture attributed to G.~Higman (cf.\ \cite{higman})
asserts that the number of conjugacy classes of a Sylow $p$-subgroup
of $\GL_n(q)$ is given by a polynomial in $q$ with integer
coefficients. This has been verified by computer calculation by
A.~Vera-L\'opez and J.~M.~Arregi for $n \le 13$, see
\cite{veralopezarregi}.  There has been much interest in this
conjecture, for example from G.~R.~Robinson \cite{robinson} and
J.~Thompson \cite{thompson}.  The reader is referred to
\cite{alperin}, \cite{evseev}, \cite{goodwinroehrle:unipotent},
\cite{goodwinroehrle:parabolic}, \cite{goodwinroehrle:calcconj}, and
\cite{goodwinroehrle:unipotent2} for recent related results.

In this note we consider the following question, which is precisely
Higman's conjecture when $P = B$ is a Borel subgroup of $\GL_n(q)$.

\begin{ques}
\label{Q:number} Let $P$ be a parabolic subgroup of $\GL_n(q)$ and
let $U = O_p(P)$. As a function of $q$, is the number $k(U)$ of
conjugacy classes of $U$ a polynomial in $q$?
\end{ques}

Here we recall that $O_p(P)$ is by definition the largest normal 
$p$-subgroup of $P$.
In this paper, we give an affirmative answer to Question
\ref{Q:number} in the following cases.

\begin{thm}
\label{thm:main} Let $P$ be a parabolic subgroup of $\GL_n(q)$ that
is the stabilizer of a flag in $\FF_q^n$ of length at most $5$, and
let $U = O_p(P)$.  Then, as a function of $q$, the number $k(U)$ of
conjugacy classes of $U$ is a polynomial in $q$ with integer
coefficients.
\end{thm}

We now explain the significance of the hypothesis imposed in Theorem
\ref{thm:main}.  Let $\bfP$ be a parabolic subgroup of $\GL_n(\bar
\FF_q)$ and let $\bfU$ be the unipotent radical of $\bfP$; where
$\bar \FF_q$ denotes the algebraic closure of $\FF_q$. All
instances when $\bfP$ acts on $\bfU$ with a finite number of orbits
were determined in \cite{hilleroehrle}; this is precisely the case
when $\bfP$ is the stabilizer of a flag in $\bar \FF_q^n$
of length at most $5$.  So Theorem \ref{thm:main} deals with
parabolic subgroups $P$ of $\GL_n(q)$ that correspond to parabolic
subgroups $\bfP$ of $\GL_n(\bar \FF_q)$ with a finite number of
conjugacy classes in $\bfU$.  In such cases, it is observed in
\cite[\S 4, Rem.\ 4.13]{hilleroehrle} that the parameterization of the
$\bfP$-conjugacy classes in $\bfU$ is independent of $q$: this is the
crucial point that we require for our proof of Theorem
\ref{thm:main}.

The proof of  Theorem \ref{thm:main} involves a translation of the
problem to a representation theoretic setting. More precisely, we
recall from \cite[\S 4]{hilleroehrle} that the $P$-conjugacy classes
in $U$ correspond bijectively to the so called $\Delta$-filtered
modules of a certain quasi-hereditary algebra $\CA_t$.  This allows
us to see that the parameterization of the $P$-orbits in $U$ is
independent of $q$ and that we can choose a set $\CR$ of
representatives that are matrices with entries equal to $0$ or $1$.
The other key point is that the structures of the centralizers
$C_P(x)$ and $C_U(x)$ for $x \in \CR$ do not depend on $q$; this is
covered in Propositions \ref{prop:C_P} and \ref{prop:C_U}.

\smallskip

We now discuss some natural generalizations of Theorem
\ref{thm:main}. First consider the case of a normal subgroup $\bfN$
of $\bfP$ with $\bfN \sub \bfU$. Still assuming that there is only a
finite number of $\bfP$-orbits in $\bfU$, we readily derive from the
proof of Theorem \ref{thm:main} that $k(U,N)$, the number of
$U$-conjugacy classes in $N = \bfN \cap U$, is given by a polynomial
in $q$ with integer coefficients.  It should also be possible to
prove that the number $k(U,N)$ is a polynomial in $q$ with just the
assumption that there are finitely many $\bfP$-orbits in $\bfN$. For
example, for $\bfN = \bfU^{(l)}$ the $l$th member of the descending
central series of $\bfU$, there is a classification of all instances
when $\bfP$ acts on $\bfU^{(l)}$ with a finite number of orbits, see
\cite{BH}.  In such situations a generalization of the proof of
Theorem \ref{thm:main} would require detailed knowledge of the
$\bfP$-conjugacy classes in $\bfN$.

\smallskip

Further, it is natural to consider the generalization of Question
\ref{Q:number}, where $\GL_n(q)$ is replaced by any finite reductive
group $G$; and also to consider the number $k(P,U)$ of $P$-conjugacy
classes in $U$ rather than $k(U)$. (In order to avoid degeneracies in the
Chevalley commutator relations, it is sensible to only consider
these generalizations when $q$ is a power of a good prime for $G$.)

At present there are no known examples, where $k(U)$ is not given by
a polynomial in $q$, and there are many cases not covered by Theorem
\ref{thm:main}, where $k(U)$ is given by a polynomial in $q$, see
for example \cite{goodwinroehrle:calcconj} and
\cite{veralopezarregi}. However, it is not necessarily the case that
$k(P,U)$ is a polynomial in $q$: indeed in \cite[Exmp.\ 4.6]{Gzeta},
it shown that in case $G$ is of type $G_2$, and $P = B$ is a Borel
subgroup of $G$, the number $k(B,U)$ is given by two different
polynomials depending on the residue of $q$ modulo 3.

Let $\bfP$ be a parabolic subgroup of a reductive algebraic group
$\bfG$ defined over $\FF_q$, and suppose that $\bfP$ has finitely
many conjugacy classes in $\bfU$; let $P$ and $U$ be the groups of
$\FF_q$-rational points of $\bfP$ and $\bfU$ respectively. Given the
discussion after Theorem \ref{thm:main} a natural generalization to
consider is whether the number $k(U)$ of conjugacy classes of $U$ is
a polynomial in $q$.  Our proof of Theorem \ref{thm:main} is
dependent on the detailed information about the $P$-conjugacy
classes in $U$. For this reason the argument does not adapt to the
case, where $G$ is any finite reductive group.  The main difficulty
is that is it is not clear whether the parameterization of
$\bfP$-orbits in $\bfU$ and the structure of centralizers depends on
the characteristic of the underlying ground field. Another problem
is that centralizers $C_\bfP(u)$ for $u \in U$ need not be
connected, so determining the $P$-classes in $U$ from the
$\bfP$-classes in $\bfU$ may be non-trivial.

\section{Translation to representation theory}
\label{sec:translation}

In this section we recall the relationship between adjoint orbits of
parabolic subgroups and modules for a certain quasi-hereditary
algebra that was established in \cite[\S 4]{hilleroehrle}.  This
relationship is central to our proof of Theorem \ref{thm:main}.  In
particular, it is crucial for Propositions \ref{prop:C_P} and
\ref{prop:C_U}, which describe the structure of certain
centralizers.  Throughout this section we work in generality over
any field, before specializing to finite fields for the proof of
Theorem \ref{thm:main} in Section \ref{sec:proof}.

\smallskip

Let $K$ be any field, and let $n, t \in \ZZ_{\ge 1}$.  Let $\d =
(d_1,\dots,d_t) \in \ZZ_{\ge 0}^t$ with $d_i \le d_{i+1}$ and  $d_t
= n$. We define the parabolic subgroup $P(\d) = P_K(\d)$ of
$\GL_n(K)$ to be the stabilizer of the flag $0 \subseteq K^{d_1}
\subseteq K^{d_2} \subseteq \ldots \subseteq K^{d_t}$ in $K^n$; any
parabolic subgroup of $\GL_n(K)$ is conjugate to $P(\d)$ for some
$\d$.  We write $U(\d) = U_K(\d) = \{u \in \GL_n(K) \mid (u-1)V_i
\subseteq V_{i-1} \text{ for each $i$}\}$ for the unipotent radical
of $P(\d)$, and $\u(\d) = \u_K(\d) = \{x \in \M_n(K) \mid xV_i
\subseteq V_{i-1} \text{ for each $i$}\}$ for the Lie algebra of
$U(\d)$. Then $P(\d)$ acts on $\u(\d)$ via the adjoint action: $g
\cdot x =  gxg^{-1}$ for $g \in P(\d)$ and $x \in \u(\d)$.  For $x
\in \u(\d)$, we write $P \cdot x$ for the adjoint $P$-orbit of $x$
and $C_P(x)$ for the centralizer of $x$ in $P$; we define $U \cdot
x$ and $C_U(x)$ analogously.

Though we are primarily interested in the conjugacy classes of
$U(\d)$ and the $P(\d)$-conjugacy classes in $U(\d)$, it is more
convenient to consider the adjoint $P(\d)$-orbits in $\u(\d)$.  The
map $x \mapsto 1 + x$ is a $P(\d)$-equivariant isomorphism between
$\u(\d)$ and $U(\d)$, which means that the adjoint $P(\d)$-orbits in
$\u(\d)$ are in bijective correspondence with the $P(\d)$-conjugacy
classes in $U(\d)$; this allows us to work with the adjoint orbits.

\smallskip

The quiver $\CQ_t$ is defined to have vertex set $\{1,\dots,t\}$;
there are arrows $\alpha_i : i \to i+1$ and $\beta_i : i+1 \to i$
for $i = 1,\dots,t-1$. Below, in Figure \ref{F:Q}, we give an
example of a quiver $\CQ_t$.

\begin{figure}[h!]
\begin{picture}(320,30)
\put(0,15){\circle*{5}} %
\put(0,5){\makebox(0,0){\small 1}} %
\put(80,15){\circle*{5}} %
\put(80,5){\makebox(0,0){\small 2}} %
\put(160,15){\circle*{5}} %
\put(160,5){\makebox(0,0){\small 3}} %
\put(240,15){\circle*{5}} %
\put(240,5){\makebox(0,0){\small 4}} %
\put(320,15){\circle*{5}} %
\put(320,5){\makebox(0,0){\small 5}} %
\put(5,10){\vector(1,0){70}} %
\put(40,0){\makebox(0,0){\small $\alpha_1$}} %
\put(85,10){\vector(1,0){70}} %
\put(120,0){\makebox(0,0){\small $\alpha_2$}} %
\put(165,10){\vector(1,0){70}} %
\put(200,0){\makebox(0,0){\small $\alpha_3$}} %
\put(245,10){\vector(1,0){70}} %
\put(280,0){\makebox(0,0){\small $\alpha_4$}} %
\put(75,20){\vector(-1,0){70}} %
\put(40,30){\makebox(0,0){\small $\beta_1$}} %
\put(155,20){\vector(-1,0){70}} %
\put(120,30){\makebox(0,0){\small $\beta_2$}} %
\put(235,20){\vector(-1,0){70}} %
\put(200,30){\makebox(0,0){\small $\beta_3$}} %
\put(315,20){\vector(-1,0){70}} %
\put(280,30){\makebox(0,0){\small $\beta_4$}} %
\end{picture}
\caption{The quiver $\CQ_5$} \label{F:Q}
\end{figure}

\noindent Let $I_t = I_{t,K}$ be the ideal of the path algebra
$K\CQ_t$ of $\CQ_t$ generated by the relations:
\begin{equation}
\label{e:relns} \beta_1 \alpha_1 =0 \text{ and }  \alpha_i \beta_i=
\beta_{i+1} \alpha_{i+1} \text{ for } i = 1,\dots,t-2.
\end{equation}
The algebra $\CA_t = \CA_{t,K}$ is defined to be the quotient
$K\CQ_t/I_t$.

Recall that an $\CA_t$-module $M$ is determined by a family of
vector spaces $M(i)$ over $K$ for $i = 1,\dots,t$ such that $M =
\bigoplus_{i=1}^t M(i)$, and linear maps $M(\alpha_i) : M(i) \to
M(i+1)$ and $M(\beta_i) : M(i+1) \to M(i)$ for $i = 1,\dots,t-1$
that satisfy the relations \eqref{e:relns}. The dimension vector
$\ddim M \in \ZZ_{\ge 0}^t$ of an $\CA_t$-module is defined by
$\ddim M = (\dim M(1),\dots,\dim M(t))$.

\smallskip

Let $\CM_t = \CM_{t,K}$ be the category of $\CA_t$-modules $M$ such
that $M(\alpha_i)$ is injective for all $i$.  Write $\CM_t(\d) =
\CM_{t,K}(\d)$ for the class of modules in $\CM_t$ with dimension
vector $\d$. It is shown in \cite[\S 4]{hilleroehrle} that the
orbits of $P(\d)$ in $\u(\d)$ are in bijective correspondence
with the isoclasses in $\CM_t(\d)$. Moreover, there is a unique
structure of a quasi-hereditary algebra on $\CA_t$ such that $\CM_t$
is the category of \emph{$\Delta$-filtered} $\CA_t$-modules, see \cite[\S
4]{hilleroehrle} and \cite[\S 6--7]{dlabringel}.

Suppose for this paragraph that $K$ is infinite.  Using the above
correspondence from \cite[\S 4]{hilleroehrle} and the results from
\cite{dlabringel}, it was proved in \cite[Thm.\ 4.1]{hilleroehrle}
that there is a finite number of $P(\d)$-orbits in $\u(\d)$ if and
only if $t \le 5$. More precisely, this is deduced from the fact
that $\CA_t$ has finite $\Delta$-representation type if and only if
$t \le 5$, see \cite[Prop.\ 7.2]{dlabringel}.

Let $t \le 5$.  Because the results in \cite[\S 4]{hilleroehrle}
are proved for an arbitrary field, see \cite[Rem.\ 4.13]{hilleroehrle},
the parametrization of indecomposable
$\Delta$-filtered $\CA_t$-modules does not depend on the field $K$;
we explain this more explicitly below. Let $\{I_1,\dots,I_m\}$ be a
complete set of representatives of isoclasses of indecomposable
$\Delta$-filtered $\CA_t$-modules, and write $\d_i$ for the
dimension vector of $I_i$. Let $x_i \in \u(\d_i)$ be such that the
$P(\d_i)$-orbit of $x_i$ corresponds to the isoclass of $I_i$. As
discussed in \cite[\S 7]{hilleroehrle}, see also
\cite[Fig.\ 10]{BHRZ}, one can choose $x_i$ to be a matrix with
entries $0$ and $1$; and these matrices do not depend on $K$. In
particular, this implies that the modules $I_i$ are absolutely
indecomposable.

Another important consequence for us is the following lemma.

\begin{lem} \label{L:reps}
Assume $t \le 5$.  We may choose a set $\CR$ of representatives of
the adjoint $P(\d)$-orbits in $\u(\d)$ such that each element of $\CR$ is a
matrix with all entries equal to $0$ or $1$.  Moreover, the elements
of $\CR$ do not depend on the field $K$, i.e.\ the positions of
entries equal to $1$ do not depend on $K$.
\end{lem}

We continue to assume that $t \le 5$, and let $\d \in \ZZ_{\ge
0}^t$. Let $P = P(\d)$, $U = U(\d)$ and $x \in \u = \u(\d)$. For the
proof of Theorem \ref{thm:main} we require information about the
structure of the centralizers $C_P(x)$ and $C_U(x)$, which is given
by Propositions \ref{prop:C_P} and \ref{prop:C_U}.

Let $M$ be a $\Delta$-filtered $\CA_t$-module (with dimension vector
$\d$) whose isoclass corresponds to the $P$-orbit of $x$. Extending
the arguments of \cite[\S 4]{hilleroehrle}, one can show that the
automorphism group $\Aut_{\CA_t}(M)$ of $M$ is isomorphic to
$C_P(x)$. Below we explain the structure of $\End_{\CA_t}(M)$ and
$\Aut_{\CA_t}(M)$, this uses standard arguments that we outline here
for convenience.  We proceed to explain how $C_U(x)$ is related to
$\End_{\CA_t}(M)$.

As above, let $\{I_1,\dots,I_m\}$ be a complete set of
representatives of isoclasses of indecomposable $\Delta$-filtered
$\CA_t$-modules. We may decompose $M$ as a direct sum of
indecomposable modules
\begin{equation} \label{e:decomp}
M \cong \bigoplus_{i=1}^m n_i I_i,
\end{equation}
where $n_i \in \ZZ_{\ge 0}$. Then
\[
\End_{\CA_t}(M) \cong \bigoplus_{i,j=1}^m n_i n_j
\Hom_{\CA_t}(I_i,I_j)
\]
as a vector space and  composition is defined in the obvious way.

We observed above that $I_i$ is absolutely indecomposable, which
means that $\End_{\CA_t}(I_i)$ is a local ring, and that we have the
decomposition $\End_{\CA_t}(I_i) = K \oplus \m_i$, where $K$ is
acting by scalars and $\m_i$ is the maximal ideal. Therefore,
\[
n_i^2 \End_{\CA_t}(I_i) \cong \M_{n_i}(K) \oplus \M_{n_i}(\m_i),
\]
where $M_{n_i}(K)$ is a subalgebra and $M_{n_i}(\m_i)$ is an ideal.
In fact, we have that $\M_{n_i}(\m_i)$ is the Jacobson radical of
$n_i^2 \End_{\CA_t}(I_i)$.

Now one can see that the Jacobson radical of $\End_{\CA_t}(M)$ is
\[
J(\End_{\CA_t}(M)) \cong \bigoplus_{i=1}^m \M_{n_i}(\m_i) \oplus
\bigoplus_{i \ne j} n_i n_j \Hom_{\CA_t}(I_i,I_j).
\]
Further, there is a complement to $J(\End_{\CA_t}(M))$ in
$\End_{\CA_t}(M)$ denoted by $C(\End_{\CA_t}(M))$ with
\[
C(\End_{\CA_t}(M)) \cong \bigoplus_{i=1}^m \M_{n_i}(K).
\]
We are now in a position to describe the automorphism group
$\Aut_{\CA_t}(M)$.  We have
\[
\Aut_{\CA_t}(M) \cong U(C(\End_{\CA_t}(M))) \ltimes (1_M +
J(\End_{\CA_t}(M))),
\]
where $U(C(\End_{\CA_t}(M)))$ denotes the group of units of
$C(\End_{\CA_t}(M))$ and $1_M + J(\End_{\CA_t}(M))$ is the unipotent
group $\{1_M + \phi \mid \phi \in J(\End_{\CA_t}(M))\}$. We have
$U(C(\End_{\CA_t}(M))) \cong \prod_{i=1}^m \GL_{n_i}(K)$, and
therefore
\[
\Aut_{\CA_t}(M) \cong \prod_{i=1}^m \GL_{n_i}(K) \ltimes N,
\]
where $N$ is a split unipotent group over $K$.  By saying $N$ is a
{\em split unipotent group} we mean that $N$ has a normal series
with all quotients isomorphic to the additive group $K$.  The
dimension of $N$ is
\begin{equation} \label{e:deltaP}
\delta := \sum_{i=1}^m n_i^2 (\dim \End_{\CA_t}(I_i) - 1) + \sum_{i\ne j}
n_i n_j \dim \Hom_{\CA_t}(I_i,I_j).
\end{equation}
One can compute all $\Hom$-groups $\Hom_{\CA_t}(I_i,I_j)$ from the
underlying Auslander-Reiten quivers of $\CA_t$ in
\cite[p221--222]{dlabringel}, see also \cite[App.\ A]{BHRZ}; the
dimensions $\dim \Hom_{\CA_t}(I_i,I_j)$ are independent of $K$.
Therefore, the positive integer $\delta$ is also independent of $K$.

We said above that $\Aut_{\CA_t}(M)$ of $M$ is isomorphic to
$C_P(x)$, so we have the following proposition.

\begin{prop} \label{prop:C_P}
The Levi decomposition of $C_P(x)$ is given by
\[
C_P(x) \cong \prod_{i=1}^m \GL_{n_i}(K) \ltimes N,
\]
where $N$, the unipotent radical of $C_P(x)$, is a split unipotent
group over $K$ of dimension $\delta$.
\end{prop}

\begin{rem}
It is natural to ask whether Proposition \ref{prop:C_P} holds 
without the restriction $t \le 5$.  The arguments above do apply for 
$t > 5$ if $K$ is assumed to be algebraically closed.  It would be 
interesting to know what happens in general, and also if Corollary 
\ref{cor:order} holds for $t > 5$.
\end{rem}

We now wish to give the structure of the centralizer $C_U(x)$.  By a
further extension of the arguments in \cite[\S 4]{hilleroehrle}, one
sees that there is an isomorphism
\[
C_U(x) \cong 1_M + \End_{\CA_t}'(M),
\]
where
\[
\End_{\CA_t}'(M) := \{\phi \in \End_{\CA_t}(M) \mid \phi M(l)
\subseteq M(l-1) \text{ for all $l$}\},
\]
here we are identifying $M(l-1)$ with its image in $M(l)$ under
$M(\alpha_{l-1})$. We have that $\End_{\CA_t}'(M)$ is a nilpotent
ideal of $\End_{\CA_t}(M)$. We define
\[
\Hom_{\CA_t}'(I_i,I_j) := \{\phi \in \Hom_{\CA_t}(I_i,I_j) \mid \phi
I_i(l) \subseteq I_j(l-1) \text{ for all $l$}\}.
\]
Then we have the isomorphism
\[
\End_{\CA_t}'(M) \cong \bigoplus_{i,j=1}^m n_i n_j
\Hom_{\CA_t}'(I_i,I_j).
\]
We write
\begin{equation} \label{e:deltaU}
\delta' := \dim \End_{\CA_t}'(M) = \sum_{i,j=1}^m n_i n_j \dim
\Hom_{\CA_t}'(I_i,I_j).
\end{equation}

From the Auslander-Reiten quivers of $\CA_t$ exhibited in
\cite[p221--222]{dlabringel}, one can compute the dimensions $\dim
\Hom_{\CA_t}'(I_i,I_j)$.  These integers are independent of $K$, so
that $\delta'$ is also independent of $K$.  The above discussion
proves the following proposition, which concludes this section.

\begin{prop} \label{prop:C_U}
The centralizer $C_U(x)$ is a split unipotent group over $K$ of
dimension $\delta'$.
\end{prop}

\section{Proof of Theorem \ref{thm:main}}
\label{sec:proof}

Let $q$ be a prime power and let $K = \FF_q$ be the field of $q$
elements.  Let $t \le 5$ and let $\d \in \ZZ_{\ge 0}^t$.  Let $P =
P(\d)$, $U = U(\d)$ and $\u = \u(\d)$ be as in the previous section,
so $P$ is a parabolic subgroup of $\GL_n(q)$.

The following corollary of Propositions \ref{prop:C_P} and
\ref{prop:C_U} is a key step in our proof of Theorem \ref{thm:main}.
It follows immediately from Propositions \ref{prop:C_P} and
\ref{prop:C_U} along with the elementary fact that the order of a
general linear group over $\FF_q$ is given by a polynomial in $q$.
The positive integers in the statement are determined in
\eqref{e:decomp}, \eqref{e:deltaP} and \eqref{e:deltaU}.

\begin{cor}
\label{cor:order} Let $x \in \u$.  Then there are positive integers
$n_1,\dots,n_m$, $\delta$ and $\delta'$ independent of $q$ such that
$$
|C_P(x)| = \prod_{i=1}^m |\GL_{n_i}(q)| \cdot q^{\delta},
$$
and
$$
|C_U(x)| = q^{\delta'}.
$$
In particular, both $|C_P(x)|$ and $|C_U(x)|$ are polynomials in $q$
with integer coefficients.
\end{cor}

We are now in a position to prove Theorem \ref{thm:main}.

\begin{proof}[Proof of Theorem \ref{thm:main}]
We have to prove that $k(U)$ is given by a polynomial in $q$.  As
discussed in the previous section $k(U) = k(U,\u)$, the number of
adjoint $U$-orbits in $\u$.  We will prove that $k(U,\u)$ is a
polynomial in $q$ with integer coefficients.

We may choose a set of representatives $\CR$ of the adjoint
$P$-orbits in $\u$, as in Lemma \ref{L:reps} and consider $\CR$ to
be independent of $q$. We have
\[
k(U,\u) = \sum_{x \in \CR} k(U, P\cdot x),
\]
where $k(U, P\cdot x)$ is the number of $U$-orbits contained in
$P\cdot x$. For $x \in \u$ and $g \in P$, we have $C_U(g \cdot x) =
gC_U(x)g^{-1}$. Therefore, we get $|U\cdot x| = |U \cdot (g \cdot
x)|$ and $k(U,P \cdot x) = |P \cdot x|/|U \cdot x|$. It follows that
\[
k(U,\u) = \sum_{x \in \CR} k(U, P \cdot x)
     = \sum_{x \in \CR} \frac{|P\cdot x|}{|U\cdot x|}
        =  \frac{|P|}{|U|}\sum_{x \in \CR} \frac{|C_U(x)|}{|C_P(x)|}\\
      = |L| \sum_{x \in \CR} \frac{|C_U(x)|}{|C_P(x)|},\\
\]
where $L$ is a Levi subgroup of $P$. Since $|L|$ is a polynomial in
$q$, it follows from Corollary \ref{cor:order} and the fact that
$\CR$ is independent of $q$ that $k(U,\u) = k(U)$ is a rational function in
$q$.  Since $k(U)$ takes integer values for all prime powers,
standard arguments show that $k(U)$ is in fact a polynomial in $q$
with rational coefficients, see for example \cite[Lem.\
2.11]{goodwinroehrle:unipotent}.

Let $\bfP$ be the subgroup of $\GL_n(\bar \FF_q)$ corresponding to
$P$ and let $\bfU$ be the unipotent radical of $\bfP$.  The
\emph{commuting variety of $\bfU$} is the closed subvariety of $\bfU
\times \bfU$ defined by
\[
\CC(\bfU) = \{(u,u') \in \bfU \times \bfU \mid uu' = u'u \}.
\]
Setting $\CC(U) = \CC(\bfU) \cap (U \times U)$ and using the
Burnside counting formula we get
\begin{equation*}
|\CC(U)| = \sum_{x\in U}|C_U(x)| = |U|\cdot k(U).
\end{equation*}
Since $|U| =q^{\dim \bfU}$ and $k(U)$ is a polynomial in $q$ with
rational coefficients, so is $|\CC(U)|$. Now using the Grothendieck
trace formula applied to $\CC(\bfU)$ (see \cite[Thm.\
10.4]{dignemichel}), standard arguments prove that the coefficients
of this polynomial are integers, see for example \cite[Prop.\
6.1]{Re}. Thus, it follows that $k(U)$ is a polynomial function in
$q$ with integer coefficients, as claimed.
\end{proof}

\begin{rem}
Let $t \le 5$ and $\d,\d' \in \ZZ_{\ge 0}^t$ with $d_t = d'_t = n$.
Suppose that $P = P(\d)$ and $Q = P(\d')$ are \emph{associated parabolic
subgroups} of $\GL_n(\FF_q )$, i.e.\ $P$ and $Q$ have Levi subgroups that
are conjugate in $\GL_n(q)$. This means that there is a $\sigma \in S_n$
such that $d_i - d_{i-1} = d'_{\sigma (i)} - d'_{\sigma (i) - 1}$ for all $i
= 1,\dots,t$, with the convention that $d_0 = d'_0 = 0$. Let $U =
U(\d)$ and $V = U(\d')$.  A consequence of
\cite[Cor.~4.7]{hilleroehrle} is that the number $k(P,U)$ of
$P$-conjugacy classes in $U$ is the same as $k(Q,V)$; the reader is
referred to \cite[Cor.\ 4.8]{goodwinroehrle:unipotent} for
similar phenomena. However, it
is not always the case that the number of conjugacy classes of $U$
is the same as the number of conjugacy classes of $V$.  For example,
take $t=3$ and consider the dimension vectors $\d = (2,3,4)$ and
$\d' = (1,3,4)$. Then $P(\d)$ and $P(\d')$ are associated parabolic
subgroups of $\GL_4(q)$.
Let $U = U(\d)$ and $V = U(\d')$. Then by direct
calculation one can check that
\begin{align*}
k(U) &= (q-1)^3+6(q-1)^2+5(q-1)+1 \\
 &\ne (q-1)^4+4(q-1)^3+6(q-1)^2+5(q-1)+1 = k(V).
\end{align*}
\end{rem}

\bigskip

{\bf Acknowledgments}: This research was funded in part by EPSRC
grant EP/D502381/1.

%%%%%%%%%%%%%%%%%%%%%%%%%%%%%%%%%%%%%%%%%%%%%%%%%%%%%%%%%%%%%%%%%%%%%%
%%%%%%%%%%%%% bibliography
%%%%%%%%%%%%%%%%%%%%%%%%%%%%%%%%%%%%%%%%%%%%%%%%%%%%%%%%%%%%%%%%%%%%%%

\end{document}